\documentclass[12pt]{amsart}
\usepackage{amsmath,amssymb,multirow,color,etoolbox,verbatim}
\apptocmd{\sloppy}{\hbadness 10000\relax}{}{}

\makeatletter
\@namedef{subjclassname@2020}{%
  \textup{2020} Mathematics Subject Classification}
\makeatother

\numberwithin{equation}{section}

\newtheorem{thm}[equation]{Theorem}
\newtheorem{prop}[equation]{Proposition}
\newtheorem{lemma}[equation]{Lemma}
\newtheorem{cor}[equation]{Corollary}

\theoremstyle{definition}
\newtheorem*{rmk}{Remark}

\newtheorem{defn}[equation]{Definition}

\newcommand{\F}{\mathbb{F}}
\newcommand{\bP}{\mathbb{P}}

\DeclareMathOperator{\charp}{char}

\DeclareMathOperator{\GL}{GL}

\DeclareMathOperator{\PSL}{PSL}
\DeclareMathOperator{\PGL}{PGL}

\DeclareMathOperator{\Aut}{Aut}

\DeclareMathOperator{\Gal}{Gal}

\DeclareMathOperator{\Tr}{Tr}

\DeclareMathOperator{\Sym}{Sym}
\renewcommand{\bar}[1]{#1\llap{$\overline{\phantom{\rm#1}}$}}
\newcommand{\abs}[1]{\lvert #1 \rvert}

\usepackage[colorlinks,pdftex,bookmarks=false]{hyperref}

\begin{document}

\baselineskip=17pt

\title[Low-degree permutation rational functions]{Low-degree permutation rational functions over finite fields}

\author{Zhiguo Ding}
\address{
  Hunan Institute of Traffic Engineering,
  Hengyang, Hunan 421001 China
}
\email{ding8191@qq.com}

\author{Michael E. Zieve}
\address{
  Department of Mathematics,
  University of Michigan,
  Ann Arbor, MI 48109-1043 USA
}
\email{zieve@umich.edu}
\urladdr{http://www.math.lsa.umich.edu/$\sim$zieve/}

\date{\today}

\thanks{
The authors thank the referee for helpful comments.  The second author thanks the National Science Foundation for support under grant DMS-1601844.}

\begin{abstract}
We determine all degree-$4$ rational functions $f(X)\in\F_q(X)$ which permute $\bP^1(\F_q)$, and answer two questions of Ferraguti and Mi\-che\-li
about the number of such functions and the number of equivalence classes of such functions up to composing with degree-one rational functions.  We also determine all degree-$8$ rational functions $f(X)\in\F_q(X)$ which permute $\bP^1(\F_q)$ in case $q$ is sufficiently large, and do the same for degree $32$ in case either $q$ is odd or $f(X)$ is a nonsquare.  Further, for thousands of other positive integers $n$, for each sufficiently large $q$ we determine all degree-$n$ rational functions $f(X)\in\F_q(X)$ which permute $\bP^1(\F_q)$ but which are not compositions of lower-degree rational functions in $\F_q(X)$.  Some of these results are proved by using a new Galois-theoretic characterization of additive (linearized) polynomials among all rational functions, which is of independent interest.
\end{abstract}

\maketitle


\section{Introduction}

Let $q$ be a power of a prime $p$.  A \emph{permutation polynomial} is a polynomial $f(X)\in\F_q[X]$ for which the map $\alpha \mapsto f(\alpha)$ is a permutation of $\F_q$.  Such polynomials have been studied both for their own sake and for use in various applications.  Much less work has been done on \emph{permutation rational functions}, namely rational functions $f(X)\in\F_q(X)$ which permute $\bP^1(\F_q):=\F_q\cup\{\infty\}$.  However, the topic of permutation rational functions seems worthy of study, both because permutation rational functions have the same applications as permutation polynomials, and because of the construction in \cite{ZR} which shows how to use permutation rational functions over $\F_q$ to produce permutation polynomials over $\F_{q^2}$.  Surprisingly little is known about permutation rational functions: for instance, the recent paper \cite{FM} contains the first classification result in this subject, namely a classification of degree-$3$ permutation rational functions.  
In this paper we give a much simpler proof of this result, and also classify permutation
rational functions in many other degrees, sometimes under the assumption that 
certain additional conditions hold.  In particular, we answer \cite[Problems~9.1 and 9.2]{FM}.

Recall that the \emph{degree} of a nonconstant
rational function $f(X)\in\F_q(X)$ is $\max(\deg a,\deg b)$ for any coprime $a,b\in\F_q[X]$ such that $f=a/b$.  The statements of our results use the following terminology:

\begin{defn}
We say that nonconstant $f,g\in\F_q(X)$ are \emph{equivalent} if $f=\mu\circ g\circ\nu$ for some degree-one $\mu,\nu\in\F_q(X)$.
\end{defn}

Plainly if $f,g\in\F_q(X)$ are equivalent then $f$ permutes $\bP^1(\F_q)$ if and only if $g$ permutes $\bP^1(\F_q)$.  For completeness, we begin with the following essentially immediate result.

\begin{lemma}\label{intro2}
Every degree-one $f(X)\in\F_q(X)$ permutes\/ $\bP^1(\F_q)$.
A degree-two $f(X)\in\F_q(X)$ permutes\/ $\bP^1(\F_q)$ if and only if $q$ is even and $f(X)$ is equivalent to $X^2$.
\end{lemma}

The following result is a more conceptual version of \cite[Thms.~5.1 and 6.2]{FM}.

\begin{thm}\label{intro3}
A degree-three $f(X)\in\F_q(X)$ permutes\/ $\bP^1(\F_q)$ if and only if it is equivalent to one of the following:
\begin{enumerate}
\item $X^3$ where $q\equiv 2\pmod 3$, 
\item $\nu^{-1}\circ X^3\circ\nu$ where $q\equiv 1\pmod{3}$ and 
for some $\delta\in\F_{q^2}\setminus\F_q$ we have
$\nu(X)=(X-\delta^q)/(X-\delta)$ and
$\nu^{-1}(X)=(\delta X-\delta^q)/(X-1)$,
\item $X^3-\alpha X$ where $3\mid q$ and either $\alpha=0$ or $\alpha$ is a nonsquare in\/ $\F_q$.
\end{enumerate}
\end{thm}

\begin{rmk}
The functions $\nu,\nu^{-1}$ in case (2) satisfy $\nu^{-1}\circ\nu=X=\nu\circ\nu^{-1}$, and also $\nu(\bP^1(\F_q))$ is the set $\Lambda$ of $(q+1)$-th roots of unity in $\F_{q^2}$ \cite[Lemma~3.1]{ZR}.  Thus $\nu^{-1}\circ X^n\circ\nu$ permutes $\bP^1(\F_q)$ if and only if $X^n$ permutes $\Lambda$, i.e., $(n,q+1)=1$.  Moreover, $\nu^{-1}\circ X^n\circ\nu$ is in $\F_q(X)$ \cite{DZ}.  These functions $\nu^{-1}\circ X^n\circ\nu$ are instances of the general class of \emph{R\'edei functions}; see \cite{DZ}.  We note that \cite[Thms.~5.1 and 6.2]{FM} present the functions in (2) as rational functions whose coefficients satisfy certain conditions, and that from the presentation in \cite{FM} one would not expect these functions to have analogues in other degrees.
\end{rmk}

\begin{rmk}
The functions in (1) and (3) of Theorem~\ref{intro3} are members of well-known classes of permutation polynomials.  Specifically, $X^n$ permutes $\F_q$ when $(n,q-1)=1$.  An \emph{additive polynomial} is a polynomial of the form $\sum_{i=0}^r \alpha_i X^{p^i}$ with $\alpha_i\in\F_q$ and $p:=\charp(\F_q)$; any such polynomial permutes $\F_q$ if and only if it has no nonzero roots in $\F_q$.  (Additive polynomials are sometimes called ``linearized polynomials" or ``$p$-polynomials".)
\end{rmk}

Theorem~\ref{intro3} was proved in \cite{FM} via a long and complicated argument involving computer calculations of prime components of certain ideals, among other things.  We give two very short non-computational proofs of Theorem~\ref{intro3} using different methods than \cite{FM}; we hope that the new understanding provided by these proofs will help readers adapt our methods to address further questions.  We then treat the much more difficult case of degree-$4$ 
permutation rational functions, which requires different methods.

\begin{thm}\label{intro4}
A degree-four $f(X)\in\F_q(X)$ permutes\/ $\bP^1(\F_q)$ if and only if one of the following holds:
\begin{enumerate}
\item $q$ is odd and $f(X)$ is equivalent to \[\frac{X^4-2\alpha X^2-8\beta X+\alpha^2}{ X^3+\alpha X+\beta}\] for some $\alpha,\beta\in\F_q$ such that $X^3+\alpha X+\beta$ is irreducible in\/ $\F_q[X]$,
\item $q$ is even and $f(X)$ is equivalent to $X^4+\alpha X^2+\beta X$ for some $\alpha,\beta\in\F_q$ such that $X^3+\alpha X+\beta$ has no roots in\/ $\F_q^*$,
\item $q\le 8$ and $f(X)$ is equivalent to a rational function in Table~\emph{1} on page~\emph{\pageref{tab}}.
\end{enumerate}
Moreover, $f(X)$ is exceptional (cf.\ Definition~\emph{\ref{defexc}}) if and only if \emph{(1)} or \emph{(2)} holds.
\end{thm}

\begin{rmk}
The functions in (2) of Theorem~\ref{intro4} are additive polynomials.  The functions in (1) are a new class of
permutation rational functions (although the characteristic zero
analogue of these functions appears in \cite[Thm.~7.1]{GMS}).
At the end of Section~\ref{sec4} we will give a single form which combines (1) and (2).  We will also determine precisely when two functions appearing in the conclusion of the above result are equivalent; in particular, for any fixed odd $q$ we will show that any two functions in (1) are equivalent.  We note that the function in (1) is $4$ times the map on $X$-coordinates induced by the multiplication-by-$2$ endomorphism of the elliptic curve $Y^2=X^3+\alpha X+\beta$, and as such it is a member of a large class of permutation rational functions which are coordinate projections of elliptic curve isogenies.  We will address such functions in a subsequent paper, where we will give an alternate proof of Theorem~\ref{intro4} based on elliptic curve arguments.
\end{rmk}

\begin{rmk}
After completing this research, we learned that Hou has independently and simultaneously studied degree-$4$ permutation rational functions \cite{Hou}.  Theorem~\ref{intro4} corrects and refines Hou's result, which shows that every degree-four permutation rational function either satisfies one of (1)--(3) of Theorem~\ref{intro4} or is equivalent to one of several other classes of examples.  Our result shows that Hou's extra classes of examples are superfluous, in the sense that they are repetitions of (1)--(3).  In particular, Hou lists a family of examples over each finite field of characteristic $3$, and by our result these are all equivalent to the functions in (1).  However, we note that this equivalence is not immediate, since to write down the equivalence one must show that the union of the images of $\bP^1(\F_q)$ under the two rational functions $X^4/(X-1)$ and $X^3+X^2$ is all of $\bP^1(\F_q)$.  This fact seems nontrivial and interesting for its own sake; we will explain it elsewhere.  Also Hou shows that for $q\le 7$ any degree-four permutation rational function which does not satisfy (1) or (2) must be equivalent to one of $19$ specific functions, while our result replaces Hou's list of $19$ functions by a list of $14$ functions.  Finally, the three permutations of $\bP^1(\F_8)$ in our Table~1 are counterexamples to Hou's result.  Hou's proof is quite long and computational, involving among other things the computation of a polynomial in five variables having more than $100$ terms.  By contrast, our proof is short and conceptual, using a completely different approach based on Galois theory.  Moreover, it does not seem to be quicker to deduce Theorem~\ref{intro4} from Hou's result than to prove Theorem~\ref{intro4} directly.  We apply our Theorem~\ref{intro4} to answer \cite[Problems~9.1 and 9.2]{FM} about the number of degree-$4$ permutation rational functions in $\F_q(X)$ and the number of equivalence classes of such functions; it does not seem to be possible to use Hou's result to answer these questions.
\end{rmk}

\begin{defn}\label{defexc} A rational function $f(X)\in\F_q(X)$ is \emph{exceptional} if $f(X)$ permutes\/ $\bP^1(\F_{q^\ell})$ for infinitely many integers $\ell$.
\end{defn}

Since bijectivity of $f(X)$ on $\bP^1(\F_{q^\ell})$ implies bijectivity of $f(X)$ on $\bP^1(\F_q)$, we see that every exceptional rational function in $\F_q(X)$ permutes $\bP^1(\F_q)$.  The following quantitative converse was proved in \cite[Thm.~2.5]{GTZ}:

\begin{lemma}\label{introbd}
If $f(X)\in\F_q(X)$ has degree $n\ge 2$ and permutes\/ $\bP^1(\F_q)$, where 
\[
\sqrt{q}>2(n-2)^2+1,
\]
then $f(X)$ is exceptional.  This inequality
holds in particular when $q\ge 4n^4$.
\end{lemma}

In light of Lemma~\ref{introbd}, for any fixed $n$ the study of degree-$n$ permutation rational functions over $\F_q$ reduces to the study of exceptional rational functions for all but finitely many values of $q$.  Our next result classifies exceptional rational functions of degree $8$.  Note that any rational function equivalent to an exceptional rational function is exceptional.

\begin{thm}\label{intro8}
If $f(X)\in\F_q(X)$ is exceptional and $\deg(f)=8$ then $q$ is even and $f(X)$ is equivalent to an additive polynomial.  The same conclusion holds if $f(X)\in\F_q(X)$ is a degree-$8$ permutation rational function with $q>73^2$.
\end{thm}

The following result classifies exceptional rational functions of
degree $32$, modulo the classification of degree-$16$ exceptional rational functions in characteristic $2$.

\begin{thm}\label{intro32}
A degree-$32$ rational function $f(X)\in\F_q(X)$ is exceptional if and only if $q$ is even and $f(X)$ is equivalent to either
\begin{enumerate}
\item $g(X^2)$ for some exceptional $g\in\F_q(X)$ of degree $16$, or
\item $L_1\circ\mu\circ L_2$ for some degree-one $\mu\in\F_q(X)$ and some exceptional additive $L_1,L_2\in\F_q[X]$.
\end{enumerate}
\end{thm}

Our final results address indecomposable exceptional rational functions, which are defined as follows.

\begin{defn}
For any field $K$,
a rational function $f(X)\in K(X)$ of degree at least $2$ is \emph{indecomposable} if it cannot be written as $g(h(X))$ with $g,h\in K(X)$ of degree at least $2$.
\end{defn}

The importance of indecomposability comes from the following classical result \cite[Thm.~1]{F2}:

\begin{lemma}\label{introdec}
A rational function $f(X)\in\F_q(X)$ of degree at least $2$ is exceptional if and only if $f=g_1\circ g_2\circ\dots\circ g_r$ for some indecomposable exceptional $g_i\in\F_q(X)$.
\end{lemma}

In light of Lemma~\ref{introdec}, in order to classify exceptional rational functions it suffices to classify indecomposable exceptional rational functions.

\begin{thm}\label{intro128}
If $f(X)\in\F_q(X)$ is an exceptional rational function of degree $128$ then $q$ is even.
If in addition $f(X)$ is indecomposable then $f(X)$ is equivalent to an additive polynomial.
\end{thm}

In Theorem~\ref{introsmall} we will prove similar results for degrees that are not prime powers, which turns out to be an easier situation to address via our methods.  In the case of prime power degree, the degrees $2$, $3$, $4$, $8$, $32$, and $128$ in the above results are the only degrees for which our method of proof yields a conclusion of the desired form.  
It seems conceivable that one might be able to classify prime-degree exceptional rational functions by further developing the approach used for polynomials in \cite[Thm.~8.1]{FGS} and \cite[App.]{MW}, but we do not pursue that problem here.
However, with current techniques it seems quite difficult to classify degree-$9$ exceptional rational functions, for instance.  We note that
there exist indecomposable additive exceptional
polynomials of any prescribed prime power degree,
and there exist other types of indecomposable
exceptional rational functions in many classes of
prime power degrees.

Theorem~8.3 of \cite{FM} asserts that there are no exceptional rational functions of degree $6$ with nonzero derivative.  By Lemmas~\ref{introdec} and \ref{intro2}, a degree-$6$ exceptional $f(X)\in\F_q(X)$ is indecomposable if and only if $f'(X)\ne 0$,
so the nontrivial portion of \cite[Thm.~8.3]{FM} is the assertion that there are no indecomposable exceptional rational functions of degree~$6$.  We prove the following vast generalization of this result:

\begin{thm}\label{introsmall}
Suppose $f(X)\in\F_q(X)$ is an indecomposable exceptional rational function whose degree $n$ satisfies $n<4096$ and $n$ is not a prime power.  Then
$n\in\{28, 45, 325, 351, 496, 784, 819, 1225, 1456, 2025,\allowbreak 3321\}$.
\end{thm}

\begin{rmk}
We caution the reader that the authors of \cite{FM} overlooked essentially all related results in the literature, including those in the papers they cited, and several of their ``new" results are in fact known.  In order to help readers of \cite{FM} avoid rediscovering known results, 
in Section~\ref{corr} we correct several inaccuracies in \cite{FM} and provide the current state of knowledge on the topics addressed in \cite{FM}.
\end{rmk}

In this paper we have chosen to use elementary
arguments even when shorter arguments were possible
if one used more advanced tools.  We did this
in order to make this paper accessible to
the largest possible audience, since we hope to
entice members of the permutation
polynomial community to study 
permutation rational functions.

This paper is organized as follows.  In the next section we quickly review background material on exceptionality and monodromy groups.  In Section~\ref{sec2} we prove Lemma~\ref{intro2} and Theorem~\ref{intro3}.
In Section~\ref{secadd} we prove a characterization of additive polynomials among all rational functions, and use it to describe the indecomposable exceptional rational functions of degree $8$, $32$, and $128$.  In Section~\ref{sec4} we prove Theorem~\ref{intro4} -- which is the most difficult result in this paper, given what was known previously -- and then we use this result in Section~\ref{seccount} to answer \cite[Problems~9.1 and 9.2]{FM} about the number of degree-$4$ permutation rational functions, and the number of equivalence classes of such functions.  In Section~\ref{sec8} we prove Theorems~\ref{intro8}, \ref{intro32}, and \ref{intro128}, and finally in Section~\ref{secsmall} we prove Theorem~\ref{introsmall}.


\section{Background material}

In this section we recall some basic facts and tools, giving proofs when we cannot find suitable references.


\subsection{Separable rational functions}

\begin{defn}
For any field $K$, a rational function $f(X)\in K(X)$ is \emph{separable} if $f(X)\notin K(X^p)$ where $p:=\charp(K)$.
\end{defn}

Note that constant rational functions are not separable.  We will often use without comment the following equivalent characterizations of separable rational functions:

\begin{lemma}\label{lemsep}
Let $K$ be a field, let $x$ be transcendental over $K$, and pick $f(X)\in K(X)$ of degree $n>0$.  Then $[K(x):K(f(x))]=n$, and
the following are equivalent:
\begin{enumerate}
\item $f(X)$ is separable,
\item $f'(X)\ne 0$,
\item the field extension\/ $K(x)/K(f(x))$ is separable,
\item the numerator of $f(X)-t$ has no multiple roots in the algebraic closure of $K(t)$, where $t$ is transcendental over $K$.
\end{enumerate}
\end{lemma}

\begin{proof}
Write $f(X)=a(X)/b(X)$ with coprime $a,b\in K[X]$, and put $t:=f(x)$ and $G(X):=a(X)-tb(X)$.  Then $G(X)\in K[X,t]$ is a degree-$n$ polynomial in $X$ which is not divisible by any nonconstant polynomial in $K[X]$.  Since $G(X)$ is also a degree-$1$ polynomial in $t$, we see that $G(X)$ is irreducible
in $K[X,t]$, so by Gauss's lemma it is irreducible in $K(t)[X]$.  Since $G(x)=0$, this shows that $G(X)$ is a nonzero constant times the minimal polynomial of $x$ over $K(t)$, so since $K(t,x)=K(x)$ it follows that $[K(x):K(f(x))]=n$.
Thus (3) does not hold if and only if $G(X)$ has multiple roots in the algebraic closure $\bar {K(t)}$ of $K(t)$, or equivalently $G(X)$ and $G'(X)$ have common roots in $\bar {K(t)}$, i.e., $\gcd(G(X),G'(X))\ne 1$.  Since $\gcd(G(X),G'(X))$ divides $G(X)$ in $K(t)[X]$, and $G(X)$ is irreducible in $K(t)[X]$, it follows that (3) does not hold if and only if $\gcd(G(X),G'(X))=G(X)$.  Since $\deg(G')<\deg(G)$, the latter condition says that $G'(X)=0$, i.e., $a'(X)=tb'(X)$, which in turn says that $a'(X)=0=b'(X)$.  Writing $p:=\charp(K)$, it follows that (3) does not hold if and only if $a,b\in K[X^p]$, i.e., if and only if $f(X)$ is not separable.  Plainly if $f(X)$ is not separable then $f'(X)=0$.  Conversely, if $f'(X)=0$ then, since $f'(X)=(b(X)a'(X)-a(X)b'(X))/b(X)^2$, it follows that $b(X)a'(X)=a(X)b'(X)$; but then coprimality of $a(X)$ and $b(X)$ implies $a(X)\mid a'(X)$, so that $a'(X)=0$, and likewise $b'(X)=0$, whence $f(X)$ is not separable.
\end{proof}

The definitions immediately imply

\begin{lemma}\label{indecinsep}
If $f(X)\in\F_q(X)$ is indecomposable but not separable then $f=\mu\circ X^p$ where $p:=\charp(\F_q)$ and $\mu(X)\in\F_q(X)$ has degree $1$.
\end{lemma}

Every nonconstant $f(X)\in\F_q(X)$ has a unique expression as $f(X)=g(X^{p^r})$ where $p:=\charp(\F_q)$, $r\ge 0$, and $g(X)\in\F_q(X)$ is separable.  Since $X^p$ permutes $\F_{q^\ell}$ for all $\ell$, we see that $f(X)$ is exceptional if and only if $g(X)$ is exceptional.  Thus the study of exceptional rational functions reduces to the separable case.


\subsection{Monodromy groups}

For any field $K$ and any separable $f(X)\in K(X)$ of degree $n>1$, let $x$ be transcendental over $K$, let $\Omega$ be the Galois closure of $K(x)/K(t)$ where $t:=f(x)$, and let $L$ be the algebraic closure of $K$ in $\Omega$.  Then $A:=\Gal(\Omega/K(t))$ and $G:=\Gal(\Omega/L(t))$ are called the
\emph{arithmetic monodromy group} and the  \emph{geometric monodromy group} of $f(X)$, respectively.  Write $S$ for the set of conjugates of $x$ over $K(t)$, and put $A_1:=\Gal(\Omega/K(x))$ and $G_1:=\Gal(\Omega/L(x))$.
We will use the following classical result, in which $\bar K$ denotes an algebraic closure of $K$:

\begin{lemma}\label{basics}
With notation as above, $G$ and $A$ are transitive subgroups of $\Sym(S)$ and $G$ is a normal subgroup of $A$ with $A/G\cong\Gal(L/K)$, where in addition $\abs{S}=n$.
In particular, if $K$ is finite then $A/G$ is cyclic.  Moreover,
\begin{enumerate}
\item $f(X)$ is indecomposable if and only if $A$ is primitive (cf.\ Definition~\emph{\ref{defprim}}).
\item If $K$ is finite then the following are equivalent:
\begin{enumerate}
\item $f(X)$ is exceptional,
\item if $H(X,Y)\in K[X,Y]$ divides the numerator of $f(X)-f(Y)$, and $H(X,Y)$ is irreducible in $\bar K[X,Y]$, then $H(X,Y)$ equals $\alpha(X-Y)$ for some $\alpha\in K^*$,
\item $A_1$ and $G_1$ have exactly one common orbit on $S$.
\end{enumerate}
\end{enumerate}
\end{lemma}

\begin{defn}\label{defprim}
A subgroup $A$ of $S_n$ is \emph{primitive} if $A$ is transitive and there are no groups strictly between $A$ and the stabilizer of $1$ in $A$.
\end{defn}

\begin{proof}
By Lemma~\ref{lemsep} we have $[L(x):L(t)]=n=[K(x):K(t)]$, so since $f(x)=t$ it follows that the numerator of $f(X)-t$ is irreducible in $L(t)[X]$.  Thus the first two sentences follows from basic Galois theory.
Item (2) is a weaker version of
\cite[Thms.~4 and 5]{Co}.  
Although (1) is known, we do not know a reference proving it in this setting, so we provide a proof.
By the Galois correspondence, $A$ is primitive if and only if there are no fields strictly between $\Omega^A=K(t)$ and $\Omega^{A_1}=K(x)$.  By L\"uroth's theorem \cite[Thm.~2]{Schinzel}, every field between $K(t)$ and $K(x)$ has the form $K(h(x))$ for some $h(X)\in K(X)$, where since $t\in K(h(x))$ we have $t=g(h(x))$ with $g\in K(X)$, so that $f=g\circ h$.  Conversely, if $f=g\circ h$ with $g,h\in K(X)$ then
$[K(x):K(h(x))]=\deg(h)$ and $[K(h(x)):K(t)]=\deg(g)$.  Thus $f(X)$ is decomposable if and only if there is a field strictly between $K(t)$ and $K(x)$.
\end{proof}

\begin{rmk}
The above proof of (1) is an algebraicization of the topological proof in \cite[\S II]{Ritt} of the analogous result over the complex numbers.
\end{rmk}


\subsection{Degree-one rational functions}

We will often use the following results without explicit comment.

\begin{lemma}\label{deg1}
For any field $K$, any degree-one $\mu(X)\in K(X)$ induces a bijective function $\mu\colon\bP^1(K)\to\bP^1(K)$.
\end{lemma}

\begin{proof}
For any $\beta\in K$, the numerator of $\mu(X)-\beta$ is a nonzero polynomial whose degree is $0$ if $\beta=\mu(\infty)$ and $1$ otherwise.  Thus in any case $\beta$ has a unique $\mu$-preimage in $\bP^1(K)$.  Also the unique $\mu$-preimage of $\infty$ is $\infty$ if $\mu(X)$ is a polynomial, and otherwise is the unique root of the denominator of $\mu(X)$.
\end{proof}

\begin{lemma}\label{line}
For any field $K$, any pairwise distinct $\alpha_1,\alpha_2,\alpha_3\in\bP^1(K)$, and any pairwise distinct $\beta_1,\beta_2,\beta_3\in\bP^1(K)$, there is a unique degree-one $\mu(X)\in K(X)$ such that $\mu(\alpha_i)=\beta_i$ for each $i$.
\end{lemma}

\begin{proof}
In case $\alpha_1=\beta_1=\infty$, this follows from the fact that there is a unique line through two points in $K\times K$.  To deduce the general case,
pick degree-one $\rho,\tau\in K(X)$ with 
$\rho(\infty)=\alpha_1$ and $\tau(\beta_1)=\infty$, and note that $\mu(\alpha_i)=\beta_i$ if and only if $\nu :=\tau\circ\mu\circ\rho$ maps
the $\rho$-preimage of $\alpha_i$ to $\tau(\beta_i)$.
\end{proof}

\begin{cor}\label{muinv}
For any field $K$, and any degree-one $\mu(X)\in K(X)$, there is a unique degree-one $\mu^{-1}(X)\in K(X)$ such that $\mu^{-1}\circ\mu=X=\mu\circ\mu^{-1}$.
\end{cor}

\begin{proof}
If $\mu^{-1}\circ\mu=X$ then $\mu^{-1}$ maps $\mu(\alpha)\mapsto\alpha$ for each $\alpha\in\{0,1,\infty\}$.  Conversely, since $\deg(\mu)=1$, the function $\mu\colon\bP^1(K)\to\bP^1(K)$ is injective, so that $\mu(0)$, $\mu(1)$, and $\mu(\infty)$ are pairwise distinct.  By Lemma~\ref{line}, there is a unique degree-one $\mu^{-1}(X)\in K(X)$ such that $\mu^{-1}(\mu(\alpha))=\alpha$ for each $\alpha\in\{0,1,\infty\}$.  Then $\mu^{-1}\circ\mu$ and $\mu\circ\mu^{-1}$ are degree-one rational functions fixing $0,1,\infty$ and $\mu(0),\mu(1),\mu(\infty)$, respectively, so each of these compositions agrees with $X$ at three points and hence equals $X$.
\end{proof}

\begin{rmk}
Explicitly, if $\mu(X)=(\alpha X+\beta)/(\gamma X+\delta)$ then $\mu^{-1}(X)=(\delta X-\beta)/(-\gamma X+\alpha)$.
\end{rmk}


\subsection{Automorphisms of $K(x)$}

For any field $K$, write $\Gamma_K$ for the set of degree-one rational functions in $K(X)$, and note that $\Gamma_K$ is a group under the operation of functional composition, by 
Corollary~\ref{muinv}.  The following two results are immediate.

\begin{lemma}\label{aut1}
Let $K$ be a field, and let $x$ be transcendental over $K$.
For each $\mu(X)\in \Gamma_K$, let $\sigma_\mu\colon K(x)\to K(x)$ map $f(x)\mapsto f(\mu^{-1}(x))$.
Then $\sigma_\mu$ is in
$\Aut_K (K(x))$, and the map $\mu(X)\mapsto\sigma_\mu$ is an isomorphism $\Gamma_K\to \Aut_K (K(x))$.
\end{lemma}

In light of the above result, we will identify $\Aut_K (K(x))$ with $\Gamma_K$ whenever convenient.  We sometimes identify $\Gamma_K$ with $\PGL_2(K)$ via the following result.

\begin{lemma}\label{aut2}
For any field $K$, the map \[\begin{pmatrix}\alpha & \beta\\\gamma & \delta \end{pmatrix}\mapsto \frac{\alpha X+\beta}{\gamma X+\delta}\] induces a surjective homomorphism $\phi\colon \GL_2(K)\twoheadrightarrow \Gamma_K$ whose kernel consists of the constant multiples of the identity matrix, so that $\phi$ induces an isomorphism $\PGL_2(K)\to \Gamma_K$.
\end{lemma}


\subsection{Fixed fields and Galois closures}

\begin{lemma}\label{aut}
Let $K$ be a field, let $x$ be transcendental over $K$, and let $G$ be a finite subgroup of $\Aut_K (K(x))$.  If $f(X)\in K(X)$ has degree $\abs{G}$ and $G$ fixes $f(x)$ then the fixed field $K(x)^G$ equals $K(f(x))$.
\end{lemma}

\begin{proof}
By basic Galois theory and Lemma~\ref{lemsep} we have
\[
[K(x): K(x)^G]=\abs{G}=\deg(f)=[K(x):K(f(x))].
\]
Since $K(x)^G\supseteq K(f(x))$, it follows that $K(x)^G=K(f(x))$.
\end{proof}

\begin{lemma}\label{PGL}
Let $K$ be a field, let $x$ be transcendental over $K$,
and let $f(X)\in K(X)$ be a separable rational function of degree $n$.  Let $\Omega$ be the Galois closure of $K(x)/K(f(x))$, and let $L$ be the algebraic closure of $K$ in $\Omega$.
If the geometric monodromy group $G$ of $f(X)$ has order $n$ then $\Omega=L(x)$.
\end{lemma}

\begin{proof}
Writing $t:=f(x)$, we have
\[
[\Omega:L(t)]=\abs{G}=n=[L(x):L(t)],
\]
which implies the result since $\Omega\supseteq L(x)$.
\end{proof}


\subsection{Branch points}

\begin{defn}\label{bpdef}
For any algebraically closed field $K$ and any nonconstant $f(X)\in K(X)$ of degree $n$, a \emph{branch point} of $f(X)$ is an element $\beta\in\bP^1(K)$ which has fewer than $n$ distinct $f$-preimages in\/ $\bP^1(K)$.
\end{defn}

\begin{lemma}\label{bpfinite}
If $\beta\in K$ is a branch point of $f(X)$, then either $\beta=f(\infty)$ or $\beta=f(\gamma)$ for some $\gamma\in K$ such that $f'(\gamma)=0$.  Any separable rational function has only finitely many branch points.
\end{lemma}

\begin{proof}
Write $f(X)=a(X)/b(X)$ where $a,b\in K[X]$ are coprime.
If $\beta\in K\setminus\{f(\infty)\}$ is a branch point of $f(X)$ then $c(X):=a(X)-\beta b(X)$ is
a degree-$n$ polynomial having fewer than $n$ distinct roots, so $c(X)$ has a multiple root $\gamma$.  Thus $c'(\gamma)=0$, so that $a'(\gamma)=\beta b'(\gamma)$, whence
\begin{align*}
f'(\gamma) &= \frac{b(\gamma)a'(\gamma)-a(\gamma)b'(\gamma)}{b(\gamma)^2} = \frac{b'(\gamma)\cdot(b(\gamma)\beta - a(\gamma))}{b(\gamma)^2} \\ &= -\frac{b'(\gamma)\cdot (\beta-f(\gamma))}{b(\gamma)} = 0,
\end{align*}
where we note that $b(\gamma)\ne 0$ since $a(\gamma)/b(\gamma)=\beta\ne\infty$.
If $f(X)$ is separable then $f'(X)\ne 0$, so the numerator of $f'(X)$ is a nonzero polynomial and hence has only finitely many roots.
\end{proof}

\begin{lemma}\label{bprat}
If $f(X)\in\F_q(X)\setminus\F_q$ then the $q$-th power map permutes the branch points of $f(X)$, and also permutes the $f$-preimages of any element of\/ $\bP^1(\F_q)$.
\end{lemma}

\begin{proof}
This holds because the multiplicity of $\alpha$ as an $f$-preimage of $f(\alpha)$ equals the multiplicity of $\alpha^q$ as an $f$-preimage of $f(\alpha^q)=f(\alpha)^q$. 
\end{proof}

\section{Degree at most $3$}\label{sec2}

In this section we give quick elementary proofs of Lemma~\ref{intro2} and Theorem~\ref{intro3}.  We give two proofs of the latter result, one being a very short proof using basic Galois theory and the other a slightly longer proof using essentially nothing.

\begin{proof}[Proof of Lemma~\ref{intro2}]
Lemma~\ref{deg1} shows that degree-one
rational functions are bijective.
If $q$ is even then clearly $X^2$ permutes $\bP^1(\F_q)$, so that also any function equivalent to $X^2$ is a permutation rational function.  Conversely, we now suppose that
$f(X)\in\F_q(X)$ is a degree-$2$ permutation rational function.  Write $\alpha:=f(0)$ and $\beta:=f(\infty)$, so that $\alpha,\beta\in\bP^1(\F_q)$, and $\alpha\ne\beta$ by the permutation hypothesis.  Letting $\mu(X)\in\F_q(X)$ be a degree-one rational function which maps $\alpha$ and $\beta$ to $0$ and $\infty$, respectively, it follows that $\widehat f(X):=\mu(f(X))$ permutes $\bP^1(\F_q)$ and fixes $0$ and $\infty$.  Thus $\widehat f(X)=(X^2+\gamma X)/b(X)$ for some $\gamma\in\F_q$ and some nonzero $b(X)\in\F_q[X]$ of degree at most $1$.  By the permutation condition, the only zero of $\widehat f(X)$ in $\F_q$ is $0$, so we must have $\gamma=0$.  Likewise the only pole of $\widehat f(X)$ in $\bP^1(\F_q)$ is $\infty$, so $b(X)$ has no roots in $\F_q$ and thus $b(X)$ is constant (since $\deg(b)\le 1$).
Hence $\widehat f(X)=\delta X^2$ for some $\delta\in\F_q^*$, so $f(X)$ is equivalent to $X^2$.  Since $\widehat f(1)=\widehat f(-1)$, the permutation condition implies $1=-1$ so $q$ is even.
\end{proof}




Theorem~\ref{intro3} asserts that all degree-$3$ permutation rational functions are equivalent to $X^3$ or an additive polynomial or a member of the following class of rational functions which is not as widely known:

\begin{defn}
A \emph{R\'edei function} is a rational function $f(X)\in\F_q(X)$ for which there exist two points in $\F_{q^2}\setminus\F_q$ which each have a unique $f$-preimage in $\bar\F_q$.
\end{defn}

These were introduced in a non-conceptual way in \cite{R}.  The
following result from \cite{DZ} describes basic properties of R\'edei functions. Here $\Lambda$ is the group of $(q+1)$-th roots of unity in $\F_{q^2}$.

\begin{lemma}\label{ok}
A rational function $f(X)\in\F_q(X)$ of degree $n>0$ is a R\'edei function if and only if $f=\mu\circ X^n\circ\nu$ for some degree-one $\mu,\nu\in\F_{q^2}(X)$ such that $\mu(\Lambda)=\bP^1(\F_q)$ and $\nu(\bP^1(\F_q))=\Lambda$.  If $f(X)\in\F_q(X)$ is a R\'edei function of degree $n>0$, then the following are equivalent:
\begin{enumerate}
\item $f(X)$ permutes\/ $\bP^1(\F_q)$,
\item $f(X)$ is exceptional,
\item $(n,q+1)=1$.
\end{enumerate}
A degree-$n$ R\'edei function is indecomposable if and only if $n$ is prime.  For any $n$ and $q$, there is a unique equivalence class of
degree-$n$ R\'edei functions in\/ $\F_q(X)$, and
it includes $\nu^{-1}\circ X^n\circ\nu$ where $\nu(X)=(X-\delta^q)/(X-\delta)$ for any $\delta\in\F_{q^2}\setminus\F_q$.
\end{lemma}

We now give two proofs of Theorem~\ref{intro3}, one
assuming familiarity with basic facts about Galois theory and degree-one rational functions, and the other assuming nothing.

\begin{proof}[Proof of Theorem~\ref{intro3}]
The remarks after Theorem~\ref{intro3} show that
the functions in (1)--(3) permute $\bP^1(\F_q)$, so that also any function equivalent to one of these permutes $\bP^1(\F_q)$.
Henceforth assume that $f(X)\in\F_q(X)$ has degree $3$ and permutes $\bP^1(\F_q)$.  We may assume $q>9$, since the result is easy to verify via computer when $q\le 9$.  Thus $f(X)$ is exceptional by Lemma~\ref{introbd}.  If $f(X)$ is inseparable then $f(X)=h(X^p)$ with $p:=\charp(\F_q)$ and $h(X)\in\F_q(X)$, so that $p=3$ and $\deg(h)=1$, whence $f(X)$ is equivalent to $X^3$.  Henceforth
assume $f(X)$ is separable.  By Lemma~\ref{basics},
the arithmetic and geometric monodromy groups $A$
and $G$ of $f(X)$ are transitive subgroups of $S_3$
such that $G$ is a proper subgroup of $A$, so that $G=A_3$ and $A=S_3$.  Thus $\F_{q^2}(x)/\F_{q^2}(f(x))$ is Galois of degree $3$, with Galois group generated by an order-$3$
automorphism of $\F_{q^2}(x)$ which fixes each
element of $\F_{q^2}$ and maps $x\mapsto \mu(x)$ for some degree-one $\mu(X)\in\F_{q^2}(X)$.  It follows that $\mu(X)$ has order $3$ under composition.
Thus there is some degree-one $\nu\in\bar\F_q(X)$
such that $\theta(X):=\nu^{-1}\circ\mu\circ\nu$ is either $X+1$ (if $3\mid q$) or $\omega X$ with $\omega$ of order $3$ (if $3\nmid q$).
Since $\widehat f(x):=f\circ\nu(x)$ is fixed
by the $\bar\F_q$-automorphism $\sigma$ of $\bar\F_q(x)$ which maps $x\mapsto\theta(x)$, and the order of $\sigma$ is $3$ which equals $[\bar\F_q(x):\bar\F_q(\widehat f(x))]$, it follows
that the fixed field of $\sigma$ is $\bar\F_q(\widehat f(x))$.  Since plainly $\sigma$ fixes $g(x)$, where $g(X):=X^3-X$ if $3\mid q$
and $g(X)=X^3$ if $3\nmid q$, it follows that 
there is a degree-one $\rho(X)\in\bar\F_q(X)$ such that $\rho\circ \widehat f(X)=g(X)$.
If $g(X)=X^3-X$ then $f(X)$ has a unique branch point, and this branch point has a unique $f$-preimage, so that both the branch point and its preimage must be in $\bP^1(\F_q)$.  Up to equivalence over $\F_q$, we may assume that both of these points are $\infty$, so that $f$ is equivalent to $a(X) \circ (X^3-X)\circ b(X)$ for some degree-one $a,b\in\bar\F_q[X]$, whence $f(X)$ is equivalent to an additive polynomial over $\F_q$.
Henceforth assume $3\nmid q$, so that $f(X)$ is equivalent over $\bar\F_q$ to $X^3$, and thus $f(X)$ has exactly two branch points, each of which has a unique preimage.  If the preimages of the branch points are in $\bP^1(\F_q)$ then, up to equivalence, we may assume that the unique root of $f(X)$ is $0$ and the unique pole of $f(X)$ is $\infty$, so that $f(X)=\gamma X^3$ with $\gamma\in\F_q^*$.  Finally, suppose that $\delta\in\bar\F_q\setminus\F_q$ is a preimage of a branch point of $f(X)$.  Since the $q$-th power map permutes the set of preimages of branch points, it follows that $\delta\in\F_{q^2}$ and the preimages of branch points are $\delta$ and $\delta^q$.  Thus the branch points of $f(X)$ are $f(\delta)$ and $f(\delta^q)=f(\delta)^q$, which are in $\F_{q^2}\setminus\F_q$.  Hence $f(X)$ is a R\'edei function, so the conclusion follows
from Lemma~\ref{ok}.
\end{proof}

Our second proof of Theorem~\ref{intro3} uses the following lemma.

\begin{lemma}\label{mu3}
Let $f(X)\in\F_q(X)$ be a separable exceptional rational function of degree $3$.  Then $f(\mu(X))=f(X)$ for some degree-one $\mu(X)\in\F_{q^2}(X)\setminus\F_q(X)$ having order $3$ under composition.
\end{lemma}

\begin{proof}
Let $G(X,Y)$ be the numerator of $f(X)-f(Y)$, so that $G(X,Y)\in\F_q[X,Y]$ has $X$-degree $3$ and $Y$-degree $3$, and $G(X,Y)$ is not divisible by any nonconstant element of $\bar\F_q[X]$ or $\bar\F_q[Y]$.  By Lemma~\ref{basics}, every irreducible factor of $G(X,Y)$ in $\F_q[X,Y]$ which remains irreducible in $\bar\F_q[X,Y]$ is a constant times $X-Y$.  Note that $X-Y$ divides $G(X,Y)$, and
write $G(X,Y)=(X-Y)H(X,Y)$ where $H(X,Y)\in\F_q[X,Y]$ has $X$-degree $2$ and $Y$-degree $2$.  Thus $H(X,Y)$ must factor in $\bar\F_q[X,Y]$ as the product of two irreducible polynomials $H_1,H_2\in\bar\F_q[X,Y]$ which each have $X$-degree $1$ and $Y$-degree $1$.  Hence $H_i(X,Y)$ is the numerator of $X-\mu_i(Y)$ for some degree-one $\mu_i(X)\in\bar\F_q(X)$.  Let $\Gamma_f$ be the set of all degree-one $\nu(X)\in\bar\F_q(X)$ for which $f\circ\nu=f$, so that $\Gamma_f$ is closed under composition, and also $\Gamma_f$ is closed under the map $\psi\colon \nu(X)\mapsto \nu^{(q)}(X)$ which raises all coefficients of $\nu(X)$ to the $q$-th power.  Since $\Gamma_f$ is also the set of degree-one $\nu(X)\in\bar\F_q(X)$ for which the numerator of $X-\nu(Y)$ divides $G(X,Y)$, we have $\Gamma_f=\{X,\mu_1(X),\mu_2(X)\}$, and Lemma~\ref{basics}
implies $\Gamma_f \cap\F_q(X)=\{X\}$.  Thus if $\mu_1(X)=X$
then since $\mu_2^{(q)}(X)\in S$ we must have $\mu_2^{(q)}(X)=\mu_2(X)$, so that $\mu_2(X)\in\F_q(X)$ and hence $\mu_2(X)=X$.  In this
case the numerator of $f(X)-f(Y)$ is $(X-Y)^3$, so that every element of $\bar\F_q\setminus f(\infty)$
has a unique $f$-preimage in $\bP^1(\bar\F_q)$, and hence is a branch point of $f(X)$, which contradicts
separability by Lemma~\ref{bpfinite}.  Hence $\mu_1(X)\ne X$, so $\mu_1(X)\notin\F_q(X)$.  Since $S=\{X,\mu_1(X),\mu_2(X)\}$ is closed under $\psi$, it follows that $\psi$ interchanges $\mu_1(X)$ and $\mu_2(X)$, so that $\mu_1(X)\in\F_{q^2}(X)$.  Since $S$ is closed under composition, $\mu_1(X)$ has order $3$ under composition.
\end{proof}

\begin{proof}[Alternate proof of Theorem~\ref{intro3}]
Just as in the start of the first proof, it suffices to show that every separable exceptional $f(X)\in\F_q(X)$ of degree $3$ is equivalent to a function in (1)--(3).  By Lemma~\ref{mu3}, there is a degree-one $\mu(X)\in\F_{q^2}(X)\setminus\F_q(X)$ which has order $3$ under composition and satisfies $f(\mu(X))=f(X)$.
The numerator of $\mu(X)-X$ has degree $2$ if $\mu(\infty)\ne\infty$ and degree at most $1$ otherwise, so the set $\Delta$ of fixed points of $\mu(X)$ in $\bP^1(\bar\F_q)$ has size $1$ or $2$.  Since $\mu(X)$ has order $3$ and $f(\mu(X))=f(X)$, if $\beta\in\bP^1(\bar\F_q)$ is not fixed by $\mu(X)$ then $\beta$, $\mu(\beta)$, and $\mu(\mu(\beta))$ are three distinct $f$-preimages of $f(\beta)$, so they comprise all $f$-preimages of $f(\beta)$, whence $f(\beta)$ is not a branch point of $f(X)$, and also $f(\beta)\notin f(\Delta)$.  Hence the elements of $\Delta$ are the only $f$-preimages of $f(\Delta)$, so (since $\abs{\Delta}\le 2$) each element of
$f(\Delta)$ is a branch point of $f(X)$.  Thus $f(\Delta)$ is the set of branch points of $f(X)$, and hence is preserved by the $q$-th power map.  Moreover, if some $\delta\in \Delta$ is not in $\bP^1(\F_q)$ then, since $f(\delta^q)=f(\delta)^q$ is in $f(\Delta)$, we have $\Delta=\{\delta,\delta^q\}$; since $\delta^q$ has the same multiplicity as an $f$-preimage of $f(\delta^q)$ as does $\delta$ as an $f$-preimage of $f(\delta)$, while the elements of $\Delta$ are the only $f$-preimages of $f(\Delta)$, it follows that $\delta$ is the unique $f$-preimage of $f(\delta)$, so that $f(\delta)\in\F_{q^2}\setminus\F_q$.

First assume $\mu(X)$ has a unique fixed point $\delta$ in $\bP^1(\bar\F_q)$.  Then $\delta\in\bP^1(\F_q)$, so also
$f(\delta)\in\bP^1(\F_q)$.
Pick degree-one $\nu,\theta\in\F_q(X)$ such that $\nu(\delta)=\infty=\theta(f(\delta)$, and put
$g:=\theta\circ f\circ\nu^{-1}$ and $\rho:=\nu\circ \mu\circ\nu^{-1}$.  Then $g\circ\rho=g$ where $g(X)$ is in $\F_q(X)$ and has degree $3$, with $\infty$ being the unique $g$-preimage of $\infty$, so that $g(X)$ is a polynomial.  Also $\rho(X)$ is a degree-one rational function in $\F_{q^2}(X)\setminus\F_q(X)$ with order $3$ under composition, and $\infty$ is the unique fixed point of $\rho(X)$.  Thus $\rho(X)=X+\gamma$ for some $\gamma\in\F_{q^2}\setminus\F_q$, and we must have $\charp(\F_q)=3$ since $\rho(X)$ has order $3$.
Since $g(X)$ is a degree-$3$ polynomial satisfying $g(X+\gamma)=g(X)$ we have $g(X)=a(X) \circ (X^3-\gamma^2 X)$ for some degree-one
$a\in\F_q[X]$.  Since $g(X)\in\F_q[X]$,
this implies $\gamma^2\in\F_q$, so that $f(X)$ is equivalent (over $\F_q$) to $X^3-\delta X$ where $\alpha := \gamma^2$ is a nonsquare in $\F_q$.

Henceforth assume that $\abs{\Delta}=2$.  Now suppose $\Delta\subseteq\bP^1(\F_q)$, and pick $\nu\in\F_q(X)$ of degree one such that $\nu(\Delta)=\{0,\infty\}$.  Then $g:=f\circ\nu^{-1}$ is a degree-$3$ rational function in $\F_q(X)$, and $\rho:=\nu\circ\mu\circ\nu^{-1}$ is a degree-one rational function in $\F_{q^2}(X)\setminus\F_q(X)$ with order $3$ under composition, where $\rho$ fixes $0$ and $\infty$.  Thus
$\rho(X)=\omega X$ where $\omega\in\F_{q^2}\setminus \F_q$ has order $3$,
so since $g\circ\rho=g$ we have $g=\theta\circ X^3$ for some
degree-one $\theta\in\F_q(X)$, whence $f(X)$ is equivalent to $X^3$.  Plainly $\omega\in\F_{q^2}\setminus \F_q$ has order $3$ just when $q\equiv 2\pmod 3$.

Finally, suppose $\Delta$ contains an element outside $\bP^1(\F_q)$,
so that $\Delta=\{\delta,\delta^q\}$ for some $\delta\in\F_{q^2}\setminus\F_q$ such that $f(\delta)\in\F_{q^2}\setminus\F_q$ and $\delta$ is the unique $f$-preimage of $f(\delta)$, while also $\delta^q$ is the unique $f$-preimage of $f(\delta)^q$. Put $\nu(X):=(X-\delta^q)/(X-\delta)$ and $\theta(X):=(X-f(\delta)^q)/(X-f(\delta))$, so
that $g:=\theta\circ f\circ\nu^{-1}$ has $0$ as its unique root and $\infty$ as its unique pole, whence $g(X)=\gamma X^3$ for some $\gamma\in\F_{q^2}^*$.  Thus
$f(X)=\theta^{-1}\circ \gamma X^3\circ\nu=
\tau\circ \widehat f(X)$ where $\tau:=\theta^{-1}\circ \gamma X\circ\nu$ and $\widehat f(X):=\nu^{-1}\circ X^3\circ\nu$.
It is easy to check that $\nu(X)$ and $\theta(X)$ map $\bP^1(\F_q)$ bijectively onto the set $\Lambda$ of $(q+1)$-th roots of unity in $\F_{q^2}$.  Since $f(X)$ permutes $\bP^1(\F_q)$, it follows that $g(X)=\gamma X^3$ permutes $\Lambda$, so that $\gamma\in \Lambda$ and $(3,q+1)=1$. Since $\rho:=\nu\circ \mu\circ \nu^{-1}$ is a degree-one rational function in $\F_{q^2}(X)$ with order $3$ under composition, and $\rho$ fixes $0$ and $\infty$, we know $\rho(X) = \omega X$ where $\omega \in \F_{q^2}^*$ has order $3$, so $\charp(\F_q)\ne 3$ which implies $q\equiv 1 \pmod 3$. Thus $\tau(X)$ is a degree-one rational function in $\F_{q^2}(X)$ which permutes $\bP^1(\F_q)$, so that $\tau(X)$ takes values in $\bP^1(\F_q)$ at $0$, $1$, and $\infty$, whence $\tau(X)\in\F_q(X)$.  This implies $\widehat f(X)=\tau^{-1}\circ f$ is in $\F_q(X)$, so that $f(X)$ is equivalent to $\widehat f(X)=\nu^{-1}\circ X^3\circ\nu$. 
\end{proof}

\begin{rmk}
Our proofs of Theorem~\ref{intro3} (and also the proof in \cite{FM}) rely on computer calculations to show that there are no non-exceptional degree-$3$ permutation rational functions in $\F_q(X)$ when $q\le 9$.  This can also be shown without a computer, as follows.  For any $f(X)\in\F_q(X)$ of degree $3$, it is easy to show that the normal closure of $\F_q(x)/\F_q(f(x))$ has genus at most $1$.  We will show in a subsequent paper that there are no non-exceptional indecomposable permutation rational functions $f(X)\in\F_q(X)$ (of any degree) for which the normal closure of $\F_q(x)/\F_q(f(x))$ has genus at most $1$.
\end{rmk}

\section{Additive polynomials}\label{secadd}

In this section we prove a Galois-theoretic characterization of additive polynomials among all rational functions, and use it to describe the indecomposable exceptional rational functions in $\F_q(X)$ of degree $2^r$ 
when $r\in\{3,5,7\}$.  The results of this section are also used in our treatment of degree-$4$ exceptional rational functions in the next section.

\begin{defn}
An \emph{additive polynomial} in\/ $\F_q[X]$ is a polynomial of the form $\sum_{i=0}^r \alpha_i X^{p^i}$ with $\alpha_i\in\F_q$ and $p:=\charp(\F_q)$.
\end{defn}

For each positive integer $s$, an additive polynomial in $\F_q[X]$ induces a homomorphism from the additive group of $\F_{q^s}$ to itself.  Since a homomorphism from a finite group to itself is bijective if and only if it has trivial kernel, this yields the 
following characterization of exceptional additive polynomials.

\begin{lemma}\label{addep}
If $f(X)\in\F_q[X]$ is additive then the following are equivalent:
\begin{enumerate}
\item $f(X)$ is exceptional,
\item $f(X)$ permutes\/ $\F_q$,
\item $f(X)$ has no roots in\/ $\F_q^*$.
\end{enumerate}
\end{lemma}

The main result of this section is the following Galois-theoretic characterization of additive polynomials.

\begin{prop}\label{add}
If $f(X)\in\F_q(X)$ is separable with geometric monodromy group $G$ then the following are equivalent:
\begin{enumerate}
\item $\abs{G}=\deg(f)$, and $\deg(f)$ is a power of $\charp(\F_q)$,
\item $f(X)$ is equivalent to an additive polynomial in\/ $\F_q[X]$.
\end{enumerate}
\end{prop}

The proof uses the following easy classical result \cite[Thm.~8 of Ch.~1]{Ore}.

\begin{lemma}\label{Ore}
If $f(X)\in\F_q[X]$ is squarefree then the roots of $f(X)$ form a group under addition if and only if $f(X)$ is additive.
\end{lemma}

\begin{proof}[Proof of Proposition~\ref{add}]
It is well known that (2) implies (1), so we assume (1).
Writing $n:=\deg(f)$, we may assume $n>1$ since otherwise the result is immediate.
Write $p:=\charp(\F_q)$, let $x$ be transcendental over $\F_q$, and put $t:=f(x)$. Let $\Omega$ be the Galois closure of $\F_q(x)/\F_q(t)$, and let $\F_Q$ be the algebraic closure of $\F_q$ in $\Omega$, where $Q=q^k$.
%
%
Since $n=\abs{G}$, Lemma~\ref{PGL} implies
$\Omega=\F_Q(x)$.
Thus $G$ is a subgroup of $\Gamma:=\Aut_{\F_Q}(\F_Q(x))$, and Lemma~\ref{aut1} shows that $\Gamma$ consists of the maps $h(x)\mapsto h(\mu(x))$ with $\mu(X)\in\F_Q(X)$ of degree one.  In particular, $\Gamma$ has order $Q^3-Q$, and one Sylow $p$-subgroup of $\Gamma$ is the group $U$ consisting of the maps $h(x)\mapsto h(x+\alpha)$ with $\alpha\in\F_Q$.
Since $G$ is a $p$-group contained in the finite group $\Gamma$, basic group theory shows that $G$ is contained in a Sylow $p$-subgroup of $\Lambda$, and that any two Sylow $p$-subgroups are conjugate.  Thus there is some $\sigma\in\Lambda$ such that $\sigma^{-1}  G\sigma$ is contained in $U$.  The elements of $\widehat G:=\sigma^{-1} G\sigma$ are thus $h(x)\mapsto h(x+\alpha)$ with $\alpha$ varying over an order-$n$ subgroup $H$ of $\F_Q$.  The fixed field $\F_Q(x)^{\widehat G}$ contains $r(x):=\prod_{\tau\in \widehat G}\tau(x)=\prod_{\lambda\in H}(x+\lambda)$,
where $r(X)$ is additive by Lemma~\ref{Ore}.
Since $\deg(r)=\abs{\widehat G}$, by Lemma~\ref{aut} we have $\F_Q(x)^{\widehat G}=\F_Q(r(x))$.  
It follows that $\F_Q(x)^G=\F_Q(\sigma(r(x)))$, so if $\sigma$ maps $h(x)$ to $h(\mu(x))$ for some degree-one $\mu(X)\in\F_Q(X)$ then $\F_Q(x)^G=\F_Q(r(\mu(x)))$.
Since $\F_Q(x)^G=\F_Q(t)=\F_Q(f(x))$, this implies $f=\nu\circ r\circ\mu$ for some degree-one $\nu\in\F_Q(X)$.  Since $\infty$ is the unique element of $\bP^1(\bar\F_q)$ which has a unique $r$-preimage, $\nu(\infty)$ is the unique element of $\bP^1(\bar\F_q)$ which has a unique $f$-preimage.  But since $f(X)\in\F_q(X)$, also $\nu(\infty)^q$ has a unique $f$-preimage, so that $\nu(\infty)^q=\nu(\infty)$ and thus $\nu(\infty)\in\bP^1(\F_q)$.  Likewise the unique $f$-preimage of $\nu(\infty)$ is $\mu^{-1}(\infty)$, which must equal its $q$-th power and hence lies in $\bP^1(\F_q)$.  Thus there are degree-one $\rho,\theta\in\F_q(x)$ such that
$\rho(\nu(\infty))=\infty$ and $\theta(\infty)=\mu^{-1}(\infty)$.  Since $\infty$ is fixed by the degree-one rational functions
$\rho\circ\nu$ and $\mu\circ\theta$ in $\F_Q(X)$, it follows that $\rho\circ\nu$ and $\mu\circ\theta$ are degree-one polynomials.  Since $r(X)$ is additive, this implies that $g:=(\rho\circ \nu)\circ r\circ(\mu\circ \theta)$ is an additive polynomial plus a constant.  But $g(X)$ equals $\rho\circ f\circ\theta$, so since $\rho,f,\theta\in\F_q(X)$ it follows that $g(X)=\alpha+L(X)$ for some $\alpha\in\F_q$ and some additive $L(X)\in\F_q[X]$.  Thus $f=\rho^{-1}\circ (X+\alpha)\circ L\circ\theta^{-1}$ is equivalent to an additive polynomial.
\end{proof}

\begin{rmk}
The polynomial case of Proposition~\ref{add} implies the main portions of \cite[Thm.~3]{BM}, \cite[Thm.~1.1]{Cfac}, and \cite[Thm.~2.10(a)]{T}.
\end{rmk}

\begin{lemma}\label{2gps}
If $K$ is an algebraically closed field with $\charp(K)\ne 2$, and $x$ is transcendental over $K$, then $\Aut_K (K(x))$ has no subgroup isomorphic to $(C_2)^3$, and every subgroup of $\Aut_K (K(x))$ isomorphic to $(C_2)^2$ is conjugate to $\langle \sigma,\tau\rangle$ where $\sigma(x)=-x$ and $\tau(x)=1/x$.
\end{lemma}

\begin{proof}
Let $G$ be a subgroup of $\Gamma:=\Aut_K (K(x))$ such that $G\cong (C_2)^r$ with $r\in\{2,3\}$.
As in Lemma~\ref{aut1}, we identify $\Gamma$ with the group $\Gamma_K$ of degree-one rational functions in $K(X)$ under the operation of functional composition.
Pick some $\sigma\in G$ of order $2$.  Then $\sigma\in K(X)$ is a degree-one rational function having order $2$ under composition, so $\sigma\ne X$.
Since the numerator of $\sigma-X$ has degree $2$ if $\sigma(\infty)\ne\infty$ and degree at most $1$ otherwise, $\sigma$ has either one or two fixed points in $\bP^1(K)$.  If $\sigma$ has just one fixed point then some conjugate of $\sigma$ in $\Gamma_K$ has $\infty$ as its unique fixed point, and hence equals $X+\alpha$ with $\alpha\in K^*$, so the order of $\sigma$ is $\charp(K)$ and hence is not $2$, contradiction.  Thus $\sigma$ has two fixed points, so there exists $\mu\in \Gamma_K$ which maps these fixed points to $0$ and $\infty$.  Then $\widehat\sigma:=\mu\sigma\mu^{-1}$ fixes $0$ and $\infty$, and has order $2$, so $\widehat\sigma=-X$.  Note that $\widehat\sigma$ is in $\widehat G:=\mu G\mu^{-1}$, which is isomorphic to $(C_2)^r$.
Since $\widehat G$ is abelian, each $\tau\in \widehat G\setminus\langle\widehat\sigma\rangle$ must commute with $\widehat\sigma$, and hence must permute the fixed points $0,\infty$ of $\widehat\sigma$.  If $\tau$ fixes $0$ and $\infty$ then $\tau=\alpha X$ with $\alpha\in K^*$, and since $\tau$ has order at most $2$ it follows that $\alpha\in\{1,-1\}$ so $\tau\in\langle\widehat\sigma\rangle$, contradiction.
Thus $\tau$ interchanges $0$ and $\infty$, so $\tau=\beta/X$ with $\beta\in K^*$.  Pick one such $\tau$, and note that any $\rho\in \widehat G\setminus\langle\widehat\sigma\rangle$ must have the form $\rho(X)=\gamma/X$ with $\gamma\in K^*$, so that $\widehat G$ contains $\tau\rho=\beta\gamma^{-1}X$, whence $\beta\gamma^{-1}\in\{1,-1\}$ so that $\rho\in\langle\widehat\sigma,\tau\rangle$.  Thus $\widehat G=\langle\widehat\sigma,\tau\rangle$, so that $r=2$.  Finally, for $\nu:=\delta X$ where $\delta^2=\beta$, we have
$\nu^{-1}\widehat G\nu=
\langle \nu^{-1}\widehat\sigma\nu,\,\nu^{-1}\tau\nu\rangle=
\langle -X,1/X\rangle$.
\end{proof}

\begin{prop}\label{even}
If $f(X)\in\F_q(X)$ is an indecomposable exceptional rational function of degree $2^r$ with $r\in\{3,5,7\}$ then $q$ is even and $f(X)$ is equivalent to an additive polynomial.
\end{prop}

\begin{proof}
Since $f(X)$ is indecomposable, it must be separable by Lemma~\ref{indecinsep}.  Let $A$ and $G$ be the arithmetic and geometric monodromy groups of $f(X)$, so Lemma~\ref{basics} implies that $A$ is a primitive subgroup of $S_n$ with $n:=2^r$, $G$ is a transitive normal subgroup of $A$ with cyclic quotient, and the one-point stabilizers $A_1$ and $G_1$ have a unique common orbit.  By testing all normal subgroups $G$ of all primitive subgroups $A$ of $S_n$, we find that the only possibility is that $G\cong (C_2)^r$.  If $\charp(\F_q)=2$ then 
Proposition~\ref{add} implies that $f(X)$ is equivalent to an additive polynomial.  Henceforth assume $\charp(\F_q)>2$.
Let $x$ be transcendental over $\F_q$, write $t:=f(x)$, let $\Omega$ be the Galois closure of $\F_q(x)/\F_q(t)$, and let $\F_Q$ be the algebraic closure of $\F_q$ in $\Omega$, where $Q=q^k$.  
Lemma~\ref{PGL} implies $\Omega=\F_Q(x)$, so $G$ is a subgroup of $\Aut_{\F_Q}(\F_Q(x))$, which in turn embeds into $\Aut_{\overline\F_Q}(\bar\F_Q(x))$.  But $G\cong (C_2)^r$ with $r>2$, contradicting Lemma~\ref{2gps}.
\end{proof}

\begin{rmk}
We do not know whether Proposition~\ref{even} remains true for larger values of $r$.  Our proof for $r\le 7$ does not by itself imply the result for $r=9$ or $r=11$, since in those cases there exist groups $A$ and $G$ satisfying the conditions used in the proof, but for which $G$ is not $(C_2)^r$.
\end{rmk}

\section{Permutation rational functions of degree $4$}
\label{sec4}

In this section we prove Theorem~\ref{intro4}.  

\begin{lemma}\label{comp}
Suppose $q$ is odd and $f(X)\in\bar\F_q(X)$ is a separable degree-$4$ rational function whose geometric monodromy group $G$ is isomorphic to $(C_2)^2$.  Then $f=\mu\circ (X^2+X^{-2})\circ\nu$ for some degree-one $\mu,\nu\in\bar\F_q(X)$.
The branch points of $f(X)$ are $\mu(\infty)$, $\mu(2)$, and $\mu(-2)$, each of which has exactly two $f$-preimages.  
\end{lemma}

\begin{proof}
Since $\abs{G}=\deg(f)$, if $x$ is transcendental over $\bar\F_q$ then the extension $\bar\F_q(x)/\bar\F_q(f(x))$ is Galois with Galois group $G$.
By Lemma~\ref{2gps}, there is some $\rho\in\Aut_{\overline\F_q}(\bar\F_q(x))$ such that $\widehat G:=\rho^{-1} G\rho$ equals $\langle\sigma,\tau\rangle$, where $\sigma(x)=-x$ and $\tau(x)=1/x$.
Here $\rho(x)=\nu(x)$ for some degree-one $\nu\in\bar\F_q(X)$.
The  fixed field $\bar\F_q(x)^{\widehat G}$ contains both $\widehat f(x):=f(\nu^{-1}(x))$ and
$g(x)$, where $g(X):=X^2+X^{-2}$.  Since $\deg(\widehat f)=\deg(g)=\abs{\widehat G}=[\bar\F_q(x):\bar\F_q(x)^{\widehat G}]$,
Lemma~\ref{aut} implies
$\bar\F_q(x)^{\widehat G}=\bar\F_q(g(x))$.  Thus $\widehat f(X) = \mu\circ g$ for some degree-one $\mu\in\bar\F_q(X)$, so that $f=\mu\circ g\circ\nu$.
For any $\alpha\in\bP^1(\bar\F_q)$, the images of $\alpha$ under the four elements of $\widehat G$ all have the same image under $g(X)$, so that if $g(\alpha)$ is a branch point of $g(X)$ then some nonidentity element of $\widehat G$ fixes $\alpha$, whence either $\alpha=-\alpha$ (so $\alpha\in\{0,\infty\}$) or $\alpha=1/\alpha$ (so $\alpha\in\{1,-1\}$) or $\alpha=-1/\alpha$ (so $\alpha\in\{i,-i\}$ where $i^2=-1$).  Since $g(X)$ induces a surjective map $\bP^1(\bar\F_q)\to\bP^1(\bar\F_q)$, it follows that every branch point of $g(X)$ is in $g(\{0,\infty,1,-1,i,-i\})=\{\infty,2,-2\}$.
Conversely it is clear that each of $\infty$, $2$, and $-2$ has two $g$-preimages.  Since $\deg(\mu)=\deg(\nu)=1$, it follows that
the branch points of $f(X)$ are $\mu(\infty)$, $\mu(2)$, and $\mu(-2)$, which are distinct and which each have two $f$-preimages.
%
%
\end{proof}

\begin{lemma}\label{odd4}
Suppose $q$ is odd and $f_1,f_2\in\F_q(X)$ are degree-$4$ exceptional rational functions. Then there exist exactly three degree-one $\mu\in\F_q(X)$ such that $\mu\circ f_1(X)$ and $f_2(X)$ have the same branch points, and for each such $\mu(X)$ there is exactly one $\nu\in\F_q(X)$ for which $\mu\circ f_1\circ\nu = f_2$.
\end{lemma}

\begin{proof}
Note that $f_1(X)$ is separable since $\charp(\F_q)\nmid\deg(f_1)$.  Let $A$ and $G$ be the arithmetic and geometric monodromy groups of $f_1(X)$, so Lemma~\ref{basics} implies $A$ is a subgroup of $S_4$, and $G$ is a transitive normal subgroup of $A$ with cyclic quotient, and the one-point stabilizers $A_1$ and $G_1$ have a unique common orbit.  It is easy to check that the only possibility is $A=A_4$ and $G=(C_2)^2$, so that $[A:G]=3$.  By Lemma~\ref{comp}, $f_1(X)$ has exactly three branch points, each of which has exactly two $f_1$-preimages in $\bP^1(\bar\F_q)$, and $f_1=\rho_1\circ g\circ\theta_1$ for some degree-one $\rho_1,\theta_1\in\bar\F_q(X)$, where $g(X):=X^2+X^{-2}$.  Since $f_1(X)$ is exceptional, in particular $f_1(X)$ permutes $\bP^1(\F_q)$.  If a branch point of $f_1(X)$ is in $\bP^1(\F_q)$ then it has two preimages in $\bP^1(\bar\F_q)$ and exactly one preimage in $\bP^1(\F_q)$; but this is impossible since the set of $f_1$-preimages of any element of $\bP^1(\F_q)$ is preserved by the $q$-th power map.  Hence no branch point of $f_1(X)$ is in $\bP^1(\F_q)$.  Since the $q$-th power map preserves the set of branch points of $f_1(X)$, it follows that the three branch points of $f_1(X)$ are $\alpha,\alpha^q,\alpha^{q^2}$ for some $\alpha\in\F_{q^3}\setminus\F_q$.  Likewise
$f_2=\rho_2\circ g\circ\theta_2$ and the branch points of $f_2(X)$ are $\beta,\beta^q,\beta^{q^2}$ for some $\beta\in\F_{q^3}\setminus\F_q$.  
Thus for any degree-one $\mu(X)\in\F_q(X)$ such that $\mu\circ f_1(X)$ and $f_2(X)$ have the same branch points, we must have $\mu(\alpha)=\beta^{q^j}$ for some $j\in\{0,1,2\}$, and then $\mu(\alpha^q)=\mu(\alpha)^q=\beta^{q^{j+1}}$ and likewise $\mu(\alpha^{q^2})=\beta^{q^{j+2}}$.
Conversely, for each $j\in\{0,1,2\}$ there is a unique degree-one $\mu_j(X)\in\bar\F_q(X)$ such that $\mu_j(\alpha^{q^i})=\beta^{q^{i+j}}$ for each $i\in\{0,1,2\}$.
Write $\mu_j(X)^q=\mu_j^{(q)}(X^q)$ where $\mu_j^{(q)}(X)$ is obtained from $\mu_j(X)$ by raising all coefficients to the $q$-th power.  Then
$\mu_j^{(q)}(\alpha^{q^i})=\mu_j(\alpha^{q^{i-1}})^q=\beta^{q^{i+j}}$, so that $\mu_j^{(q)}(X)$ and $\mu_j(X)$ take the same values as one another at each of the three elements $\alpha^{q^i}$.  It follows that $\mu_j^{(q)}(X)=\mu_j(X)$, so that $\mu_j(X)\in\F_q(X)$.  Here $\mu_j\circ f_1(X)$ and $f_2(X)$ have the same branch points, so the three functions $\mu_j(X)$ comprise all degree-one $\mu(X)\in\F_q(X)$ for which $\mu\circ f_1(X)$ and $f_2(X)$ have the same branch points.

Fix $j$ and put $\mu:=\mu_j$. 
Let $\Delta:=\{\infty,2,-2\}$ be the set of branch points of $g(X)$.
The $q$-th power map transitively permutes the set of branch points of $f_1(X)$, which is $\rho_1(\Delta)$ since $f_1(X)=\rho_1\circ g\circ\theta_1(X)$.
Upon replacing $\rho_1$ and $\theta_1$ by $\rho_1^{(q^\ell)}$ and $\theta_1^{(q^\ell)}$ for a suitable $\ell$, we may assume that $\rho_1(\infty)=\alpha$; note that this replacement preserves the identity
$f_1=\rho_1\circ g\circ \theta_1$, since $f_1^{(q^\ell)}=f_1$ and $g^{(q^\ell)}=g$.   Further, since
$-g(iX)=g(X)$ when $i^2=-1$, we may replace $\rho_1$ and $\theta_1$ by $\rho_1(-X)$ and $i\theta_1(X)$ if necessary in order to assume that $\rho_1(2)=\alpha^q$.
Likewise, we may assume that $\rho_2(\infty)=\mu(\alpha)$ and $\rho_2(2)=\mu(\alpha^q)$, so that $\rho_2(\delta)=\mu(\rho_1(\delta))$ for each $\delta\in\Delta$.  Since $\rho_2(X)$ and $\mu(\rho_1(X))$ are degree-one rational functions which agree at three points, we have $\rho_2(X)=\mu(\rho_1(X))$, whence
\[
f_2\circ\theta_2^{-1} = \rho_2\circ g =
\mu \circ f_1\circ\theta_1^{-1}.
\]
Let $S$ be the set of all degree-one $\eta\in\bar\F_q(X)$ for which $\mu\circ f_1\circ\eta=f_2$.  Then $S$ contains $\theta:=\theta_1^{-1}\circ\theta_2$, and 
$\Gal(\bar\F_q(X)/\F_q(X))$ permutes $S$.
Moreover, for any degree-one $\eta\in\bar\F_q(X)$ we have $\eta\in S$ if and only if $f_1\circ\theta\circ\eta^{-1}=f_1$, or equivalently $\theta\circ\eta^{-1}\in G$.
In particular, we have $\abs{S}=4$.
If $\eta(X)\in G$ lies in $\F_q(X)$ then, since $f_1\circ\eta=f_1$, injectivity of $f_1(X)$ on $\bP^1(\F_q)$ implies that $\eta(X)$ fixes each element of $\bP^1(\F_q)$, so that $\eta(X)$ has at least four fixed points which implies $\eta(X)=X$.  Since $\Gal(\bar\F_q(X)/\F_q(X))$ permutes $G$, and $G$ is a four-element set with a unique fixed point under this map, it follows that the three nonidentity elements of $G$ comprise a single orbit under this map, and are in $\F_{q^3}(X)\setminus\F_q(X)$.
For any $\eta\in S\setminus\{\theta\}$ we know that $\theta\circ\eta^{-1}$ is a nonidentity element of $G$, and hence is in $\F_{q^3}(X)\setminus\F_q(X)$.  Thus at least one of $\theta$ or $\eta$ is not in $\F_{q^4}(X)$.  Since $\Gal(\bar\F_q(X)/\F_q(X))$ permutes the four-element set $S$, it follows that the orbits of this action have sizes $1$ and $3$, and the size-$1$ orbit consists of the unique $\nu\in\F_q(X)$ for which $\mu\circ f_1\circ\nu=f_2$.
\end{proof}

\begin{table}[!htbp]\label{tab}
\caption{Non-exceptional degree-$4$ permutation rational functions over $\F_q$ (the stabilizer size is defined in Proposition~\ref{FM2})}
\renewcommand{\arraystretch}{2.5}
\begin{tabular}{|c|c|c|c|}
\hline
$q$&$f(X)$&Conditions&Stabilizer size \\ \hline\hline

$8$&$\displaystyle{\frac{X^4+\alpha X^3+X}{X^2+X+1}}$&$\alpha^3+\alpha=1$&6 \\ \hline

$7$&$X^4+3X$&&3\\ \hline

\multirow{2}{*}{$5$}&$\displaystyle{\frac{X^4+X+1}{X^2+2}}$&&1 \\
&$\displaystyle{\frac{X^4+X^3+1}{X^2+2}}$&&3 \\ \hline

\multirow{3}{*}{$4$}&$\displaystyle{\frac{X^4+\omega X}{X^3+\omega^2}}$&\multirow{3}{*}{$\omega^2+\omega=1$}&6\\
&$\displaystyle{\frac{X^4+X^2+X}{X^3+\omega}}$&&2\\
&$\displaystyle{\frac{X^4+\omega X^2+X}{X^3+X+1}}$&&2\\
\hline

\multirow{3}{*}{$3$}&$X^4-X^2+X$&&1\\
&$\displaystyle{\frac{X^4+X+1}{X^2+1}}$&&1\\
&$\displaystyle{\frac{X^4+X^3+1}{X^2+1}}$&&3\\ \hline

\multirow{2}{*}{$2$}&$X^4+X^3+X$&&1\\
&$\displaystyle{\frac{X^4+X^3+X}{X^2+X+1}}$&&2\\ \hline

\end{tabular}
\end{table}

\begin{proof}[Proof of Theorem~\ref{intro4}]
Pick $f(X)\in\F_q(X)$ of degree four.
If $f(X)$ is inseparable then $\charp(\F_q)=2$, and Lemma~\ref{intro2} implies $f(X)$ permutes $\bP^1(\F_q)$ if and only if $f(X)$ is equivalent to $X^2\circ\mu\circ X^2$ for some degree-one $\mu\in\F_q(X)$.  Since $X^2\circ\mu=\nu\circ X^2$ for some degree-one $\nu\in\F_q(X)$, this says $f(X)$ is equivalent to $X^4$.  Henceforth assume $f(X)$ is separable.

If $f(X)$ permutes $\bP^1(\F_q)$ but $f(X)$ is not exceptional, then Proposition~\ref{introbd} implies $q\le 81$.  For $q\le 81$, a computer search shows that 
Table~1 contains representatives for all equivalence classes of non-exceptional degree-4 bijective rational functions.

It remains to determine the exceptional functions $f(X)$.  Let $A$ and $G$ be the arithmetic and geometric monodromy groups of $f(X)$, so Lemma~\ref{basics} implies $A$ is a subgroup of $S_4$, and $G$ is a transitive normal subgroup of $A$ with cyclic quotient, where the one-point stabilizers $A_1$ and $G_1$ have a unique common orbit.  By inspection, the only possibility is $A=A_4$ and $G=(C_2)^2$.  If $q$ is even then Proposition~\ref{add} implies $f(X)$ is equivalent to an additive polynomial, and by Lemma~\ref{addep} this polynomial is exceptional just when it has no roots in $\F_q^*$.

Henceforth assume $q$ is odd.
By Lemma~\ref{odd4}, there is at most one
equivalence class of degree-$4$ exceptional rational functions over $\F_q$, so it remains
only to show that the functions in (1) of Theorem~\ref{intro4} are exceptional.
Pick $\alpha,\beta\in\F_q$ for which $X^3+\alpha X+\beta$ is irreducible, and put
\[ f(X):=\frac{X^4-2\alpha X^2-8\beta X+\alpha^2}{X^3+\alpha X+\beta}.\]
Let $\gamma_1,\gamma_2,\gamma_3$ be the roots  of $X^3+\alpha X+\beta$, so that $\gamma_i\in\F_{q^3}\setminus\F_q$.
Then the numerator of $f(X)-f(Y)$ is
\[
(X-Y) \prod_{i=1}^3 \bigl(XY-\gamma_i(X+Y)-\alpha-2\gamma_i^2\bigr).
\]
Since each of the above factors has $Y$-degree $1$, and plainly the numerator of $f(X)-f(Y)$ has no factor of the form $Y-\delta$ with $\delta\in\bar\F_q$, we see that each factor is irreducible in $\bar\F_q[X,Y]$.
But none of the factors other than $X-Y$ is a constant multiple of a polynomial in $\F_q[X,Y]$, so that $f(X)$ is exceptional by Lemma~\ref{basics}.
\end{proof}

\begin{rmk}
The polynomials $f(X)$ in case (1) of Theorem~\ref{intro4} have the unusual property that the branch points of $g(X):=f(X)/4$ are precisely the three elements of $g^{-1}(\infty)\setminus\{\infty\}$, namely the three roots of $X^3+\alpha X+\beta$.
\end{rmk}

\begin{rmk}
Cases (1) and (2) of Theorem~\ref{intro4} may be combined into a single case which covers both even and odd characteristic.  Namely, it is easy to deduce from Theorem~\ref{intro4} that a separable $f(X)\in\F_q(X)$ of degree $4$ is exceptional if and only if $f(X)$ is equivalent to
\[
\frac{X^4-2\alpha X^2-8\beta X+\alpha ^2}{X^3+\alpha X+\beta}
\]
for some $\alpha,\beta\in\F_q$ such that $X^3+\alpha X+\beta$ is irreducible in $\F_q[X]$.
\end{rmk}


\section{The number of degree-$4$ permutation rational functions}\label{seccount}

In this section we answer \cite[Problems~9.1 and 9.2]{FM},
which for each $q$ ask for the number of equivalence classes of degree-$4$ permutation rational functions over $\F_q$, an explicit representative for each class, and the total number of degree-$4$ permutation rational functions over $\F_q$.

\begin{prop} The number of equivalence classes of exceptional $f(X)\in\F_q(X)$ of degree $4$ is
\begin{enumerate}
\item $1$ if $q$ is odd,
\item $(q+4)/3$ if $q\equiv 2\pmod{6}$,
\item $(q+8)/3$ if $q\equiv 4\pmod{6}$.
\end{enumerate}
If $q$ is odd then any function in (1) of Theorem~\ref{intro4} represents the unique equivalence class.  If $q$ is even then a system of distinct representatives for the classes consists of $X^4$, all polynomials $X^4+X^2+\alpha X$ with $\alpha\in\F_q\setminus\{\beta^3+\beta\colon\beta\in\F_q\}$, and if $q\equiv 4\pmod{6}$ then in addition $X^4+\gamma X$ and $X^4+\gamma^2 X$ for a single prescribed non-cube $\gamma\in\F_q^*$.  The number of equivalence classes of non-exceptional degree-$4$ permutation rational functions $f(X)\in\F_q(X)$ is
\begin{enumerate}
\item $0$ if $q>8$,
\item $1$ if $q=7$,
\item $2$ if $q\in\{2,5\}$,
\item $3$ if $q\in\{3,8\}$,
\item $5$ if $q=4$.
\end{enumerate}
Representatives for the distinct classes are all the entries in Table~1, using all possible values for $\alpha$ and $\omega$, except that for each of the two choices of $\omega$ the third entry for $q=4$ yields the same equivalence class.
\end{prop}

\begin{proof}
The non-exceptional case follows from Theorem~\ref{intro4} and routine computations.
The exceptional case for $q$ odd follows from Theorem~\ref{intro4} and Lemma~\ref{odd4}.  Now suppose $q$ is even.  By Theorem~\ref{intro4}, each exceptional $f(X)\in\F_q(X)$ of degree $4$ is equivalent to an additive polynomial $L(X)$ with no roots in $\F_q^*$.  If $L(X)$ is inseparable then $L=h(X)\circ X^2$ for some $h(X)\in\F_q[X]$ of degree $2$, and plainly $h(0)=0$; since $L(X)$ has no roots in $\F_q^*$, also $h(X)$ has no roots in $\F_q^*$, so that $h(X)$ is a monomial and thus $L(X)$ is equivalent to $X^4$.  If $L(X)$ is separable then, by composing on both sides with polynomials of the form $\delta X$ with $\delta\in\F_q^*$, we see that $L(X)$ is equivalent to either $X^4+X^2+\alpha X$ with $\alpha\in\F_q^*$ or $X^4+\gamma X$ with $\gamma$ varying over a set of coset representatives for $\F_q^*/(\F_q^*)^3$.  The condition that these polynomials have no roots in $\F_q^*$ says that $\alpha\in \Delta:=\F_q\setminus\{\beta^3+\beta\colon\beta\in\F_q\}$ and $\gamma\notin(\F_q^*)^3$.  Here $\abs{\Delta}=\lfloor{(q+1)/3}\rfloor$ by \cite[Prop.~4.6]{Tval}.  It remains only to show that equivalent separable degree-$4$ monic additive
polynomials $f,g$ with degree-$2$ coefficient in $\{0,1\}$ are equal.  So suppose that $g=\mu\circ f\circ\nu$ for some degree-one $\mu,\nu\in\F_q(X)$.  Since $g^{-1}(\infty)=\{\infty\}=f^{-1}(\infty)$, and any element of $\bar\F_q$ has four distinct preimages under each of $f(X)$ and $g(X)$, we see that both $\mu(X)$ and $\nu(X)$ must fix $\infty$, and hence must be degree-one polynomials.  If neither $f(X)$ nor $g(X)$ has a degree-$2$ term then the ratio of their degree-$1$ coefficients is a cube so $f(X)=g(X)$.  If at least one of $f(X)$ or $g(X)$ has a degree-$2$ term then equating coefficients of $X^4$ and $X^2$ shows that $\nu(X)$ and $\mu(X)$ are monic, and then since $f(X)$ and $g(X)$ are additive it follows that $f(X)=g(X)$.
\end{proof}

Problem 9.2 in \cite{FM} asks for an explicit formula for the number of degree-$4$ permutation rational functions over $\F_q$.  To this end, let $\Gamma$ be the group of degree-one rational functions over $\F_q$ under the operation of functional composition.  Then the equivalence classes of nonconstant rational functions over $\F_q$ are precisely the orbits of $\Gamma\times \Gamma$ on the set $\F_q(X)\setminus\F_q$ under the group action $(\mu(X),\nu(X))\colon g(X)\mapsto \mu\circ g\circ\nu^{-1}$.
By the orbit-stabilizer theorem, in order to compute the size of an equivalence class it suffices to compute the size of the stabilizer of any prescribed element of the class.

\begin{prop}\label{FM2}
If $f(X)\in\F_q(X)$ is a degree-$4$ permutation rational function, then the number of pairs $(\mu(X),\nu(X))$ of degree-one rational functions in\/ $\F_q(X)$ for which $\mu\circ f\circ\nu=f$ is as follows:
\begin{enumerate}
\item $3$ if $q$ is odd and $f(X)$ is exceptional,
\item $q$ if $q$ is even and $f(X)$ is equivalent to an additive polynomial having a degree-$2$ term,
\item $3q$ if $q\equiv 4\pmod{6}$ and $f(X)$ is equivalent to $X^4+\gamma X$ with $\gamma\in\F_q^*$,
\item $q^3-q$ if $q$ is even and $f(X)$ is equivalent to $X^4$,
\item the stabilizer size listed in the corresponding entry of Table~\emph{1}, if $f(X)$ is equivalent to a rational function in Table~\emph{1}.
\end{enumerate}
\end{prop}

\begin{proof}
If $f(X)$ is non-exceptional, and hence is equivalent to a function in Table~1, then this is a simple computation.  Now assume $f(X)$ is exceptional.  
If $q$ is odd then the result follows from Lemma~\ref{odd4} by putting $f_1=f_2=f$.
Henceforth assume $q$ is even, so that $f(X)$ is equivalent to an additive polynomial having no roots in $\F_q^*$.  Since equivalent functions yield the same number of pairs $(\mu(X),\nu(X))$, we may assume that $f(X)$ is a monic additive polynomial.  If $f(X)=X^4$ then for any degree-one $\nu(X)\in\F_q(X)$ there is a unique degree-one $\mu(X)\in\F_q(X)$ such that $\mu\circ f\circ\nu=f$, so there are $q^3-q$ pairs $(\mu,\nu)$.  If $f(X)\ne X^4$ then $f^{-1}(\infty)=\{\infty\}$, while any element of $\F_q$ has at least two distinct $f$-preimages in $\bar\F_q$.  Thus if degree-one $\mu,\nu\in\F_q(X)$ satisfy $\mu\circ f\circ\nu=f$ then both $\mu(X)$ and $\nu(X)$ must fix $\infty$, so that $\mu(X)=\gamma X+\delta$ and $\nu(X)=\alpha X+\beta$ with $\alpha,\gamma\in\F_q^*$ and $\beta,\delta\in\F_q$.  By equating coefficients, we see that $\mu\circ f\circ\nu$ equals $f(X)$ if and only if $\delta=\gamma f(\beta)$, $\gamma=1/\alpha^4$, and $\alpha^{4-j}=1$ for each $j$ occurring as the degree of a term of $f(X)$.  This yields the stated formulas.
\end{proof}

\begin{rmk}
In case $q$ is odd and $f(X)$ is given by (1) of Theorem~\ref{intro4}, one can explicitly describe the pairs $(\mu(X),\nu(X))$ with $\mu\circ f\circ\nu=f$.
Let $\gamma$ be a root of $X^3+\alpha X+\beta$, so that $\gamma_1:= \gamma \in\F_{q^3}\setminus\F_q$, and put $\gamma_2:=\gamma^q$ and $\gamma_3:=\gamma^{q^2}$.  It suffices to describe the pairs with $\mu(4\gamma_1)=4\gamma_2$, since the other pairs are then $(\mu\circ\mu,\,\nu\circ\nu)$ and $(X,X)$.  Put
\[
\mu(X) := 4\frac{-(3\beta+\delta)X+4\alpha^2}{3\alpha X+24\beta-4\delta}
\]
where
\[
\delta:=\gamma_1^2 \gamma_2 + \gamma_2^2 \gamma_3 + \gamma_3^2 \gamma_1 \quad\text{ and }\quad
\epsilon:=\gamma_1 \gamma_2^2 + \gamma_2 \gamma_3^2 + \gamma_3 \gamma_1^2.
\]
Writing $\Tr$ for the trace map from $\F_{q^3}$ to $\F_q$, let $\nu(X)$ be the rational function
\[
\frac{\bigl(-(\epsilon+3\beta)+\Tr((\gamma_1-\gamma_2)^{\frac{3q^2+2q+1}2})\bigr)X + \alpha^2 - \alpha\Tr((\gamma_1-\gamma_2)^{\frac{q^2+2q+1}2})}
{3\alpha X+\delta+3\beta +\Tr((\gamma_1-\gamma_2)^{\frac{q^2+2q+3}2})}. 
\]
Then $\mu(X),\nu(X)\in\F_q(X)$ satisfy $\mu(4\gamma_1)=4\gamma_2$ and $\mu\circ f\circ\nu = f$.
\end{rmk}

Formulas for the number of degree-$4$ permutation polynomials over $\F_q$ follow immediately from the previous two results, via the orbit-stabilizer theorem.

\begin{cor}\label{dumb}
The number of degree-$4$ permutation rational functions over\/ $\F_q$ is as follows:
\begin{enumerate}
\item $(q^3-q)^2/3$ if $q$ odd and $q>7$,
\item $q(q-1)(q+2)(q^3+1)/3$ if $q$ even and $q>8$,
\item $222768$ if $q=8$,
\item $75264$ if $q=7$,
\item $24000$ if $q=5$,
\item $8160$ if $q=4$,
\item $1536$ if $q=3$,
\item $78$ if $q=2$.
\end{enumerate}
\end{cor}

\begin{rmk}
Of all the counting results in this section, we find that the most illuminating one is Proposition~\ref{FM2}, about the size of the stabilizer of a permutation rational function under the stated group action.
\end{rmk}



\section{Exceptional rational functions of degrees $8$, $32$, and $128$}\label{sec8}

In this section we prove Theorems~\ref{intro8}, \ref{intro32}, and \ref{intro128}.

\begin{proof}[Proof of Theorem~\ref{intro8}]
It suffices to prove the first sentence of Theorem~\ref{intro8}, since this implies the second sentence by Lemma~\ref{introbd}.
Let $f(X)\in\F_q(X)$ be exceptional of degree $8$.
If $f(X)$ is indecomposable then the conclusion follows from Proposition~\ref{even}.
Thus we may assume $f(X)=g(h(X))$ with $g,h\in\F_q(X)$ of degree at least $2$.  Plainly $g(X)$ and $h(X)$ are exceptional.  The degrees of $g(X)$ and $h(X)$ are $2$ and $4$ in some order, which by Lemma~\ref{intro2} implies that $q$ is even and one of $g(X)$ and $h(X)$ is equivalent to $X^2$.  It follows that $f'(X)=0$, so that $f(X)$ is inseparable and thus $f(X)=u(X^2)$ with $u(X)\in\F_q(X)$ of degree $4$, where again $u(X)$ is exceptional.
By Theorem~\ref{intro4}, $u(X)=\mu\circ L\circ\nu$ where $\mu,\nu\in\F_q(X)$ have degree one and $L(X)=X^4+\alpha X^2+\beta X$ with $\alpha,\beta\in\F_q$ where $L(X)$ has no roots in $\F_q^*$.  Then $\nu\circ X^2=X^2\circ\rho$ for some degree-one $\rho(X)\in\F_q(X)$, so that $f(X)=\mu\circ L\circ X^2\circ\rho$, whence $f(X)$ is equivalent to $L(X^2)$ as desired.
\end{proof}

\begin{proof}[Proof of Theorem~\ref{intro32}]
Let $f(X)\in\F_q(X)$ be an exceptional rational function of degree $32$.
Proposition~\ref{even} implies the conclusion if $f(X)$ is indecomposable, so we may assume that $f=g\circ h$ with $g,h\in\F_q(X)$ both of degree at least $2$.  Then $g(X)$ and $h(X)$ are exceptional.  If either $g(X)$ or $h(X)$ has degree $2$ then Lemma~\ref{intro2} implies $q$ is even and either $g'(X)=0$ or $h'(X)=0$, so that $f'(X)=0$, whence $f=u(X^2)$ where $u(X)\in\F_q(X)$ is exceptional.  Finally, if $\{\deg(g),\deg(h)\}=\{4,8\}$ then Theorem~\ref{intro8} implies $q$ is even, so Theorems~\ref{intro4} and \ref{intro8} imply that both $g(X)$ and $h(X)$ are equivalent to additive polynomials.
\end{proof}

\begin{proof}[Proof of Theorem~\ref{intro128}]
Let $f(X)\in\F_q(X)$ be exceptional of degree $128$.  By Proposition~\ref{even}, if $f(X)$ is indecomposable then $q$ is even and $f(X)$
is equivalent to an additive polynomial.  Now assume that $f=g\circ h$ for some $g,h\in\F_q(X)$ of degree at least $2$.  Then $g,h$ are exceptional.  Since one of $g(X)$ and $h(X)$ has degree $2^r$ with $r\in\{1,3,5\}$, the combination of Lemma~\ref{intro2} and Theorems~\ref{intro8} and \ref{intro32} implies that $q$ is even.
\end{proof}


\section{Indecomposable exceptional rational functions of non-prime power degree}\label{secsmall}

In this section we prove Theorem~\ref{introsmall}.

\begin{lemma}\label{AnSn}
If $f(X)\in\F_q(X)$ is a separable exceptional rational function of degree $n\ge 5$, then the arithmetic monodromy group of $f(X)$ is not in $\{A_n,S_n\}$.
\end{lemma}

\begin{proof}
Let $A$ and $G$ be the arithmetic and geometric monodromy groups of $f(X)$, and suppose that $A\in\{A_n,S_n\}$.
Since $n\ge 5$, we know that $A_n$ is simple.
Since $G$ is transitive, we have $\abs{G}\ge n$.
Next, $G_0:=G\cap A_n$ satisfies $[G:G_0]\le [A:A_n]\le 2$, so that $\abs{G}\le 2\abs{G_0}$ and thus $\abs{G_0}\ge\abs{G}/2\ge n/2>1$.  Since $G$ is normal in $A$, also $G_0$ is normal in $A_n$, so that $G_0=A_n$.  But the stabilizer of $1$ in $A_n$ has orbits $\{1\}$ and $\{2,3,\dots,n\}$, so these are also the orbits of the stabilizer of $1$ in each of $A$ and $G$.  This contradicts exceptionality by Lemma~\ref{basics}.
\end{proof}

\begin{proof}[Proof of Theorem~\ref{introsmall}]
Suppose $f(X)\in\F_q(X)$ is an indecomposable exceptional rational function of degree $n$, where $n<4096$ and $n$ is not a prime power. 
Then $f(X)$ is separable
by Lemma~\ref{indecinsep}.
Let $A$ and $G$ be the arithmetic and geometric monodromy groups of $f(X)$, so that $A$ is a subgroup of $S_n$ and $G$ is a transitive normal subgroup of $A$ with cyclic quotient. 
By Lemma~\ref{basics}, indecomposability of $f(X)$ implies that $A$ is primitive, and exceptionality implies that if $A_1$ and $G_1$ are the subgroups of elements of $A$ and $G$ which fix the element $1$ of $\{1,2,\dots,n\}$ then $A_1$ and $G_1$ have a unique common orbit on $\{1,2,\dots,n\}$.  Since $n>1$ (by indecomposability) and $n$ is not a prime power, we have $n\ge 6$, so Lemma~\ref{AnSn} implies $A\notin\{A_n,S_n\}$.  Finally, for each candidate $n$, we use the following Magma code to test whether there are groups satisfying the conditions required of $A$ and $G$.
\begin{verbatim}
for n in [2..4095] do
  if not IsPrimePower(n) then
    for A in PrimitiveGroups(n) do
      if A notin {Alt(n),Sym(n)} then A1:=Stabilizer(A,1);
        for GG in NormalSubgroups(A) do G:=GG`subgroup;
          if #G gt 1 and G ne A and IsCyclic(A/G) and 1 eq
            #{i:i in Orbits(Stabilizer(G,1))|i in Orbits(A1)}
            and IsTransitive(G) then n; continue n;
          end if;
        end for;
      end if;
    end for;
  end if;
end for;
\end{verbatim}
\end{proof}

\begin{rmk}
The reason Theorem~\ref{introsmall} restricts to $n<4096$ is that Magma's database of primitive
groups consists of the primitive groups
of degree less than $4096$.  If this database were
extended to larger degrees then one could extend
Theorem~\ref{introsmall} to larger degrees via the
same proof.
\end{rmk}


\section{Corrections to the paper \cite{FM}} \label{corr}

The paper \cite{FM} did not accurately describe the literature on this topic, which might hinder readers of \cite{FM} who would like to contribute to this topic.  For the benefit of such readers, we now correct several inaccuracies in \cite{FM}.

Although \cite[p.~867]{FM} says ``there is a lack of references that deal, in a compact and self-contained way, with the finite field theoretic framework", in fact this is done in \cite{Co,FriedMc,FGS,GMS,GTZ,LMT,TSchur} among many other sources.

Although the explicit production of all rational functions with prescribed monodromy groups is discussed in \cite[p.~868]{FM}, no references are given, suggesting that \cite{FM} was the first paper to do this in any situation.  In fact, this was done previously in
\cite{FGS,GRZ,GZ,K,MW,T,Zariski,Zariski2}, via an assortment of powerful methods.

The authors of \cite{FM} assert that Theorem 3.2 of their paper is a generalization to rational functions of a result which had been proved previously only for polynomials.  However, this is not the case -- instead, \cite[Thm.~3.2]{FM} is a weaker version of \cite[Thms.~4 and 5]{Co}, which appeared $50$ years before \cite{FM}.  In the intervening $50$ years the latter result has been generalized in
\cite[Thm.~1 and Prop.~1]{FriedMc},
\cite[General Exceptionality Lemma, p.~185]{FGS}, and \cite[Thm.~1.1]{GTZ} to more general settings than rational functions.
Although \cite[p.~871]{FM} asserts that Cohen and Fried only proved this result for polynomials, in fact neither Cohen nor Fried has written any paper proving these results in just the polynomial case (the authors of \cite{FM} say their proofs follow ideas of Cohen, but
the only paper of Cohen's they cite is a survey paper which does not contain any proofs, and the authors of \cite{FM} overlooked all historical comments in that survey paper, including the fact that the ``new" rational function analogues in \cite{FM} were well known).

The result \cite[Lemma~4.1]{FM} is a weaker version of a very special case of \cite[Thm.~2.5]{GTZ}; we note that the proof of \cite[Lemma~4.1]{FM} uses among other things a different result from \cite{GTZ}.

The paper \cite{FM} neglects to mention that major progress towards classifying groups satisfying conditions (1)--(3) on \cite[p.~872]{FM} was made in \cite{F,FGS} and especially \cite{GMS}.  



\normalsize

\end{document}